\documentclass[10pt]{amsart}

\usepackage{amsthm}
\usepackage{amssymb}
\usepackage{amsmath}
\usepackage{mathtools}
\usepackage{pdfpages}
\usepackage{float}
\usepackage[colorlinks=true,pdftex,unicode=true,linktocpage,bookmarksopen,hypertexnames=false]{hyperref}
\usepackage{tikz-cd}
\usepackage{caption}
\usepackage[capitalise]{cleveref}
\usepackage{xspace}

\DeclareMathOperator{\Aut}{Aut}

\DeclareMathOperator{\Ann}{Ann}
\DeclareMathOperator{\Fix}{Fix}

\DeclareMathOperator{\Ker}{Ker}

\DeclareMathOperator{\Soc}{Soc}

\DeclareMathOperator{\op}{op}

\newcommand{\PS}{\textnormal{(}$\operatorname{S}$\textnormal{)}\xspace}
\newcommand{\ps}{\textnormal{(}$\operatorname{s}$\textnormal{)}\xspace}
\newcommand{\PBS}{\textnormal{(}$\operatorname{BS}$\textnormal{)}\xspace}

\makeatletter
\numberwithin{equation}{section}
\numberwithin{figure}{section}
\numberwithin{table}{section}
\newtheorem{thm}{Theorem}[section]
\newtheorem*{thm*}{Theorem}
\newtheorem{lem}[thm]{Lemma}
\newtheorem{cor}[thm]{Corollary}
\newtheorem{pro}[thm]{Proposition}

\theoremstyle{definition} 
\newtheorem{defn}[thm]{Definition}
\newtheorem{notation}[thm]{Notation}

\newtheorem{question}[thm]{Question}

\newtheorem{rem}[thm]{Remark}
\newtheorem{exa}[thm]{Example}

\usepackage{academicons}
\definecolor{orcidlogocol}{HTML}{A6CE39}

\makeatother

\title[Derived-indecomposable solutions of the YBE]{On derived-indecomposable solutions of the Yang--Baxter equation}
\author{I. Colazzo}
\author{M. Ferrara}
\author{M. Trombetti}

\address{Department of Mathematics, College of Engineering, Mathematics and Physical Sciences, University of Exeter, Exeter EX4 4QF, UK\\
ORCID: 0000-0002-2713-0409}
\email{I.Colazzo@exeter.ac.uk}

\address{Dipartimento di Matematica e Fisica, Università degli Studi della Campania “Luigi Vanvitelli”, viale Lincoln 5, Caserta, Italy\\
ORCID: 0000-0001-6002-327X}
\email{maria.ferrara1@unicampania.it}

\address{Dipartimento di Matematica e Applicazioni, Università degli Studi di Napoli Federico II, Complesso Universitario Monte S. Angelo, Via Cintia, 80126, Naples, Italy
\\
ORCID: 0000-0003-4532-3690}
\email{marco.trombetti@unina.it}

\begin{document}

\maketitle

\begin{abstract}
If $(X,r)$ is a finite non-degenerate set-theoretic solution of the Yang--Baxter equation, the additive group of the structure skew brace $G(X,r)$ is an $FC$-group, \hbox{i.e.} a group whose elements have finitely many conjugates. Moreover, its multiplicative group is virtually abelian, so it is also close to being an $FC$-group itself. If one additionally assumes that the derived solution of $(X,r)$ is indecomposable, then for every element $b$ of $G(X,r)$ there are finitely many elements of the form $b*c$ and $c*b$, with $c\in G(X,r)$. 
This naturally leads to the study of a brace-theoretic analogue of the class of $FC$-groups. For this class of skew braces, the fundamental results and their connections with the solutions of the YBE are described: we prove that they have good torsion and radical theories, and that they behave well with respect to certain nilpotency concepts and finite generation. 
\end{abstract}

\medskip
\medskip

81R50, 20F24, 08A05.

\section{Introduction}

\noindent The Yang--Baxter equation (YBE) is one of the fundamental equations of physics. It takes its name from the independent works of the physicists Chen-Ning Yang \cite{Yang} and Rodney Baxter \cite{Baxter}. It turns out that this equation plays a relevant role in many different subjects such as knot theory, braid theory, operator theory, Hopf algebras, quantum groups, $3$-manifolds and the monodromy of differential equations. 

A {\it solution} of the YBE is a pair $(V,R)$, where $V$ is a vector space and $R$ is a linear map $R:\, V\otimes V\longrightarrow V\otimes V$ such that $$(R\otimes\operatorname{id})(\operatorname{id}\otimes R)(R\otimes\operatorname{id})=(\operatorname{id}\otimes R)(R\otimes\operatorname{id})(\operatorname{id}\otimes R).$$ At the present time, we are far from being able to provide a full classification of the solutions of the YBE. However, in recent years, there has been an increasing interest in the so-called {\it set-theoretic} (or {\it combinatorial}) solutions of the YBE, \hbox{i.e.} those solutions $(V,R)$ such that $R$ is induced by linear extensions of a bijective map $$r:\, X\times X\longrightarrow X\times X,$$ where $X$ is a basis of $V$ (see \cite{Drinfeld}); in this case, also the pair $(X,r)$ is called a {\it set-theoretic} (or {\it combinatorial}) solution.

Finding all solutions of the YBE is hard, but one strategy to tackle this problem is to study solutions that cannot be further reduced to smaller ones (see \cite{ESG01}). A solution $(X,r)$  is \emph{decomposable} if there is a partition of $X$ into two disjoint non-empty subsets $X_1$ and $X_2$ such that
$r(X_1 \times X_1)\subseteq X_1 \times X_1$, and $r(X_2 \times X_2)\subseteq X_2 \times X_2$; otherwise $(X,r)$ is \emph{indecomposable}.
Many authors have focused on studying and classifying such solutions with a special emphasis on those which are also {\it involutive} and {\it non-degenerate}, i.e. solutions $(X,r)$ such that $r^2=\operatorname{id}$, and if we write $$r(x,y)=(\sigma_x(y),\tau_y(x))$$ then the maps $\sigma_x$ and $\tau_y$ are bijective, for all $x,y\in X$; see \cite{camp2021criterion, castelli2021classification,CaMaSt22,CaPiRu20,ESG01,JePiZD22,LeRaVe2022x,ramirez2021decomposition,Ru05,Ru20,Ru22,SmSm18}. On the other hand, almost nothing is known about indecomposable solutions which are non-degenerate and non-involutive (see \cite{Cedo,LeVe19,Ra22x}). A particular family of non-degenerate and non-involutive indecomposable solutions is that given by those whose derived solution is indecomposable; in the following, we refer to these as {\it derived-indecomposable solutions} (see \cref{defderivedindecomposable}).

The aim of this paper is to provide algebraic tools to study derived-in\-de\-com\-po\-sable solutions. This is accomplished through the study of the structure skew brace $G(X,r)$ associated with a solution $(X,r)$ of the YBE.

A skew brace is essentially a set endowed with two group structures, linked together with a ``distributivity-like'' relation (see \cite{Okninski,OriginalBrace,Rump2} and next section for precise definitions). If $(X,r)$ is a non-degenerate solution, the skew brace controlling its structure is the {\it structure skew brace} $G(X,r)$ \cite{DopoRump,OriginalBrace,Soloviev}.
Understanding the structure of $G(X,r)$ makes it possible to better understand the structure of the corresponding solution, and this is the reason why the basic theory of skew braces has recently been the subject of many papers (see, for example, \cite{Bonatto,MR4457900,Cedo,MR4256133,JeVAV22x} and the references therein). Of course, if $(X,r)$ is a solution and $X$ is finite, then the multiplicative group and the additive group of $G(X,r)$ are finitely generated, but in general they are both infinite groups. However, in this case, it also turns out that the multiplicative group is virtually abelian (see, for instance, \cite[Corollary 7.2]{LeVe19}), while the additive group is finite over its centre (see \cite[Theorem 2.7]{structureMonoid}). Central-by-finite groups  are very special cases of an intensively studied class of groups: the class of {\it $FC$-groups}, \hbox{i.e.} the class of groups whose elements have finitely many conjugates (see the monograph \cite{Tomkinson}). This class of groups generalises the class of abelian groups and that of finite groups, and in fact it shares many good properties with these two classes of groups. Due to its nice behaviour, the class of $FC$-groups has often been employed to study very difficult problems in various contexts. For example, Seghal, Zassenahaus \cite{Zassenhaus} and Polcino Milies \cite{pulcino,pulcino2,pulcino1} gave a detailed description of the group rings whose group of invertible elements is an $FC$-group.

Starting from the above remarks, one naturally defines classes of skew braces that behave like $FC$-groups (see Definitions~\ref{deffc} and \ref{defbfc}) and it turns out they control derived-indecomposable solutions of the YBE (see Theorems~\ref{indecomposableprima} and \ref{finalthmBrace}). In addition, 
this class of skew braces shares similarities with the classes of finite and trivial skew braces (see Theorems~\ref{periodicreduction} and \ref{indecomposableprima}), and although it is much bigger, it can still be described quite well (compare with, for instance, our Theorem~\ref{BaerThm} and \cite[Theorem 5.4]{MR4256133}). Moreover, here one can deal with finite generation (see Theorem~\ref{theoremfg}), torsion concepts (see Theorem~\ref{theoperiodicfc}), radical theory (see Theorem \ref{radical}) and (under some mild additional hypotheses) with certain nilpotency concepts (see, for instance, Theorem~\ref{LN}). 

The layout of the paper is the following. In Section~\ref{secpre}, we give basic results concerning skew braces and $FC$-groups. In Section~\ref{fcgroups}, we introduce and describe skew braces with property \PS and their connections with the YBE. In Section~\ref{bfc}, we prove an analogue of a theorem of Bernard Neumann for skew braces having the property \PBS and show how these skew braces naturally arise in the context of derived-indecomposable solutions of the YBE.

\section{Definitions and preliminaries}\label{secpre}

\noindent In this section we give the necessary background on skew braces and on $FC$-groups.

\subsection{Skew braces}

\noindent Let $B$ be a set. If $(B,+)$ and $(B,\circ)$ are (not necessarily abelian) groups, then the triple $(B,+,\circ)$ is a {\it skew \textnormal(left\textnormal) brace} if the {\it skew \textnormal(left\textnormal) distributive law} $$a\circ(b+c)=a\circ b-a+a\circ c$$ holds for all $a,b,c\in B$. Now, let $(B,+,\circ)$ be a skew brace. We refer to $(B,+)$ as the \emph{additive group} of $B$ and to $(B,\circ)$ as the \emph{multiplicative group} of $B$. We denote by $0$ the identity of $(B,+)$, by $1$ the identity of $(B,\circ)$,  by $-a$ and $a^{-1}$ the inverses of $a$ in $(B,+)$ and $(B,\circ)$, respectively, and, if $n \in \mathbb{N}$, then $na = a+a+ \cdots + a$ ($n$ times) and $a^{n}=a\circ a \circ \cdots \circ a$ ($n$~times). The skew distributive law easily implies $0=1$. It should be also noticed that the map $$\lambda:\, a\in B\mapsto \left(\lambda_a:\, b\mapsto \lambda_a(b)=-a+a\circ b\right)\in\Aut(B,+)$$ is a group homomorphism from $(B,\circ)$ to $\Aut(B,+)$ and the following relations hold $$a+b=a\circ\lambda_a^{-1}(b),\quad a\circ b=a+\lambda_a(b),\quad -a=\lambda_a\left(a^{-1}\right).$$ 

\indent In analogy with ring theory, a third relevant (not necessarily associative) operation in skew braces is defined as follows $$a\ast b=\lambda_a(b)-b=-a+a\circ b-b$$ and one can easily check that it satisfies the relations 
\[
\begin{array}{c}
a \ast (b + c) = a \ast b + b + a \ast c - b,\\[0.2cm]
(a \circ b) \ast c=a \ast (b \ast c) + b \ast c + a \ast c,
\end{array}
\] 
for all $a,b,c\in B$. If the additive group of $B$ is abelian, we simply call $B$  a {\it \textnormal{(}left\textnormal{)} brace}, or, a {\it skew \textnormal{(}left\textnormal{)} brace of abelian type}. If $(G,\cdot)$ is any group, then $(G,\cdot,\cdot)$ is a skew brace called a {\it trivial skew brace} and $(G,\cdot^{\op},\cdot)$ is a skew brace called an \emph{almost trivial skew brace}; if $(G,\cdot)$ is abelian, then the trivial skew brace and the almost trivial skew brace coincide and we simply speak of a {\it trivial brace}.

If we consider the natural semidirect product $G=(B,+)\rtimes(B,\circ)$, where \[
\begin{array}{c}
(a,b)(c,d)=(a + \lambda_b(c),b\circ d)
\end{array}
\] for all $a,b,c,d\in B$, then an easy computation shows that the operation $\ast$ corresponds to a commutator of type \[
\begin{array}{c}\label{equationfinale}
\left[(0,a),(b,0)\right]=(a\ast b,0),\tag{$\star$}
\end{array}
\] for all $a,b\in B$ (note that our convention for commutators in a group $(G,\cdotp)$ is $[x,y]=xyx^{-1}y^{-1}$).

A {\it left ideal} of a skew brace $B$ is a subgroup $I$ of $(B,+)$ such that $\lambda_a(I)\subseteq I$ for all $a\in B$; this is equivalent to $B\ast I\subseteq I$, so $I$ is also a subgroup of $(B,\circ)$. An {\it ideal} is a left ideal that is normal in $(B,+)$ and $(B,\circ)$ (note that the last condition is equivalent to demanding that $I\ast B\subseteq I$); in this case, it is known that $B/I$ is a skew brace and $a+I=a\circ I$ for all $a\in B$. A skew brace is {\it simple} if it has no proper non-zero ideals.

The {\it socle} of $B$ is defined as $\Soc(B)=\operatorname{Ker}(\lambda)\cap Z(B,+)$ and the {\it annihilator} (see \cite{catino}) of $B$ is defined as $\operatorname{Ann}(B)=\operatorname{Soc}(B)\cap Z(B,\circ)$, where $Z(B,+)$ and $Z(B,\circ)$ are the centre of $(B,+)$ and $(B,\circ)$, respectively. Moreover, we let $B^{(2)}=B\ast B$ be the subgroup of $(B,+)$ generated by all elements of the form $a\ast b$ for all $a,b\in B$. It can be proved that $\operatorname{Soc}(B)$, $\operatorname{Ann}(B)$ and $B^{(2)}$ are ideals. In connection with $B^{(2)}$, we further observe that $B$ is a trivial skew brace if and only if $B^{(2)}=\{0\}$.


Finally, we illustrate the connection between skew braces and solutions of the YBE. Let $B$ be a skew brace and let $$r_B:\, (a,b)\in B\times B\mapsto \left(\lambda_a(b),\,\lambda_a(b)^{-1}\circ a\circ b\right)\in B\times B.$$ Then $(B,r_B)$ is a non-degenerate solution of the YBE. Conversely, if $(X,r)$ is a non-degenerate solution of the YBE, then there is a unique skew brace structure over the structure group $$G(X,r)=\langle X\, |\, xy=\sigma_x(y)\tau_y(x),\; x,y\in X\rangle$$ such that $r_{G(X,r)}(\iota\times\iota)=(\iota\times\iota)r$, where $\iota:\, X\longrightarrow G(X,r)$ is the canonical map. The multiplicative group of this skew brace is $G(X,r)$ and the additive is $$A(X,r)=\langle X\,|\, x+\sigma_x(y)=\sigma_x(y)+\sigma_{\sigma_x(y)}\big(\tau_y(x)\big),\;\textnormal{for all}\; x,y\in X\rangle$$ (see \cite{Zhu,Soloviev}). We refer to this skew brace as the {\it structure skew brace} of $(X,r)$ \cite[Theorem 3.9]{OriginalBrace}.

\subsection{{\it FC}-groups}

Let $G$ be a group. An element $x$ of $G$ is an {\it $FC$-element} of $G$ if it has finitely many conjugates in $G$, or, equivalently, if $|G:C_G(x)|$ is finite; the latter condition is also equivalent to requiring that the normal core $\left(C_G(x)\right)_G$ of the centralizer $C_G(x)$ of $x$ in $G$ has finite index in $G$. The set of all $FC$-elements of $G$ is a (characteristic) subgroup of $G$ which is called the {\it $FC$-centre} of $G$ and is usually denoted by $FC(G)$. The group $G$ is an {\it $FC$-group} if $G=FC(G)$. Natural examples of $FC$-groups are restricted direct products of finite groups (unrestricted direct products of infinitely many non-abelian finite groups are not $FC$-groups); groups with a finite commutator subgroup (or, by a well-known theorem of Schur, groups which are finite over their centre); abelian groups. $FC$-groups have been introduced by Reinhold Baer, and they have been studied by many authors, including Gor\v cakov, Hall, Neumann and, more recently, Kurdachenko and Tomkinson. In this subsection we summarise some of the main results concerning the class of $FC$-groups; we refer to \cite{Tomkinson} as a general reference on the subject.

A finitely generated $FC$-group is finite over its centre. Actually, in an arbitrary $FC$-group $G$, the factor group $G/Z(G)$ is locally finite, \hbox{i.e.} every finite subset is contained in a finite subgroup; thus, it follows from a well-known theorem of Schur that $G'$ is locally finite and hence the subset of all periodic elements of $G$ is a subgroup. It turns out that every $FC$-group can be embedded in a direct product of a torsion-free abelian group and a locally finite $FC$-group, so, in order to prove many properties of $FC$-groups, it is often sufficient to prove them in the locally finite case. Furthermore, the locally finite case is a very good case, because {\it Dietzmann's lemma} holds for $FC$-groups: every finite subset of a periodic $FC$-group is contained in a finite normal subgroup.

\section{Skew braces with property (S)}\label{fcgroups}

\noindent In this section, we introduce property (S) for skew braces, and we study the basic theory of skew braces with such a property. The main results describe the finite generation (see \cref{theoremfg}) and the torsion theory (see \cref{theoperiodicfc}) of skew braces satisfying (S). Exploiting these results we highlight relations between this class of skew braces and solutions of the YBE. At the end of the section, nilpotency concepts and radical theory are both studied.

\smallskip

Let $B$ be a skew brace and let $x\in B$. Let $\Fix^r_B(x)=\{b\in B\,:\, x\ast b=0\}$; if there is no ambiguity, we simply write $\Fix^r(x)$. Clearly, $\Fix^r(x)$ is a subgroup of $(B,+)$ and its index is the cardinal number of the set $\{x \ast b\,:\, b\in B\}$ since the assignment $$b+\Fix^r(x)\mapsto x \ast b$$ defines a bijective correspondence between the set of all right cosets of $\Fix^r(x)$ in $(B,+)$ and the set $\{x \ast b\,:\, b\in B\}$. 

Similarly, we define $\Fix^l_B(x)=\Fix^l(x)=\{b\in B\,:\,b\ast x=0\}$. It is easy to see that $\Fix^l(x)$ is a subgroup of $(B,\circ)$ and its index is the cardinal number of the set $\{b\ast x\,:\, b\in B\}$ since the assignment $$b\circ\Fix^l(x)\mapsto b\ast x$$ defines a bijective correspondence between the set of all right cosets of $\operatorname{Fix}^l(x)$ in~$(B,\circ)$ and the set  $\{b\ast x\,:\, b\in B\}$.

Moreover, let $C_x^{(B,+)}=C_x^+$ and $C_x^{(B,\circ)}=C_x^\circ$ be the centralisers of $x$ in $(B,+)$ and $(B,\circ)$, respectively. 
Then $C_x^+\cap\Fix^r(x)$ is a subgroup of $(B,+)$ and $C_x^\circ\cap\Fix^l(x)$ is a subgroup of $(B,\circ)$. Finally, we define the {\it annihilator} $\Ann_B(x)$ of \emph{an element}~$x$ in $B$ as $$\Ann_B(x)=\Ann(x)=\Fix^r(x)\cap C_x^+\cap C_x^\circ\cap\Fix^l(x).$$ 

\begin{rem}\label{remref}
Note that if $c\in \Fix^l(x)\cap\Fix^r(x)$, then $c\in C_x^+$ if and only if $c\in C_x^\circ$. Thus, $\Ann(x)$ could have been defined as $\Fix^l(x)\cap\Fix^r(x)\cap C_x^+$ or as $\Fix^l(x)\cap\Fix^r(x)\cap C_x^\circ$.
\end{rem}

\begin{lem}\label{remref2}
Let $B$ be a skew brace and let $x\in B$. If $b\in C_x^\circ\cap\Fix^r(x)$, then $b\circ a\in\Fix^r(x)$ if and only if $a\in\Fix^r(x)$.
\end{lem}
\begin{proof}
Note that $$
\begin{array}{c}
x+x\ast (b\circ a)+b\circ a=x\circ b\circ a=b\circ x\circ a=b\circ x+(b\circ x)\ast a+a\\[0.2cm]
=x\circ b+(b\circ x)\ast a+a
=x+b+b\ast(x\ast a)+x\ast a+b\ast a+a\\[0.2cm]
=x+b+b\ast(x\ast a)+x\ast a-b+b\circ a
=x+b\circ(x\ast a)-b+b\circ a\\[0.2cm]
=x+b\circ (x\ast a+a).
\end{array}
$$ Thus, $x\ast a=0$ if and only if $x\ast(b\circ a)=0$ and the statement is proved.
\end{proof}

\begin{thm}\label{annsubgroup}
Let $B$ be a skew brace. If $x\in B$, then $\Ann(x)$ is a subgroup of~$(B,\circ)$.
\end{thm}
\begin{proof}
By Remark \ref{remref}, we have $$\Ann(x)=\Fix^l(x)\cap\Fix^r(x)\cap C_x^\circ.$$ Since we already know that $\Fix^l(x)$ is a subgroup of $(B,\circ)$, we only need to show that $A(x)=\Fix^r(x)\cap C_x^\circ$ is a subgroup of $(B,\circ)$. Let $b\in A(x)$. Since $b\circ b^{-1}=0\in\Fix^r(x)$, then $b^{-1}\in\Fix^r(B)$ by Lemma \ref{remref2}. Thus, $b^{-1}\in A$. Moreover, if $b,c\in A(x)$, then a further application of Lemma \ref{remref2} shows that $b\circ c\in\Fix^r(x)$. Therefore $b\circ c\in A(x)$, which shows that $A(x)$ is a subgroup of $(B,\circ)$ and completes the proof.
\end{proof}

\medskip

Note that in the case of a two-sided brace, the annihilator of an element is a subgroup of both the underlying additive and multiplicative structures.

The following three easy examples show that, in general, the annihilator of an element is not a normal subgroup of $(B,\circ)$ nor a subgroup of $(B,+)$.

\begin{exa}
    Let $(B,+)=\langle a,b,c\ | \ 3a = 2b=2c=c+b+a=0\rangle$ be the symmetric group on three elements and the
    multiplication be defined in such a way that $(B,\circ)=N\rtimes X$, where $N=\{0,a,2a\}$, $X=\{0,b\}$, $a\circ b=c$ and $2a\circ b=a+b$ (see \cite[Theorem 3.12]{AcBo20}). It is easy to check that $(B,+,\circ)$ is a skew brace. Moreover, note that also $(B,\circ)$ is the symmetric group on three elements, however $(B,+,\circ)$ is not the trivial skew brace (for example, $a+b\neq c = a\circ b$). Now, $\Ann(b)=\left\{0,b\right\}$ is not a left ideal, since it is not $\lambda$-invariant: $\lambda_a(b) = - a + a\circ b = a + a + c = a + b\notin \Ann(b)$. It is also clear that $\Ann(b)$ is a subgroup with respect both operations but it is not normal: $-a+b+a = a+b$ and $a^{-1}\circ b \circ a = c$.
\end{exa}

\begin{exa}
    Let $B$ be the trivial skew brace with additive and multiplicative groups isomorphic to the symmetric group on three elements $\langle a,b,c\ | \ 3a = 2b=2c=c+b+a=0\rangle$. One can check that $\Ann(b)=\left\{0,b\right\}$ is a subgroup of $(B,+)$ and $(B,\circ)$ and it is clearly $\lambda$-invariant, hence $\Ann(b)$ is a left ideal of $B$. On the other hand, $\Ann(b)$ is not a normal subgroup of $(B,+)$, which implies that it is not an ideal of $B$.
\end{exa}

\begin{exa}
    Let $B$ be the brace whose underlying additive group is isomorphic to \hbox{$\mathbb{Z}/2\mathbb{Z}\times \mathbb{Z}/4\mathbb{Z}$} and whose multiplication is defined by 
    \begin{align*}
        \begin{pmatrix}
            y_1\\
            z_1+2x_1 
        \end{pmatrix} \circ \begin{pmatrix} y_2\\ z_2+2x_2\end{pmatrix} = \begin{pmatrix}
        y_1 +y_2 +(x_1 +y_1 +z_1 +y_1z_1)z_2\\ 
        z_1 +2x_1 +2z_1y_2 +2(y_1 +x_1z_1)z_2 +z_2 +2x_2
        \end{pmatrix}
    \end{align*}
    for any $x_1,x_2,y_1,y_2,z_1,z_2\in\{0,1\}$ (see \cite[Theorem 3.1]{Ba18} for further detail).
    Then $\Ann(0,3) =\{(0,0), (1,1)\}$ is not a subgroup of $(B,+)$ since $(1,1)+(1,1) =(0,2)$, but is a subgroup $(B,\circ)$.
\end{exa}

\medskip

We chose to define $\Ann(x)$ as above because the annihilator of a skew brace behaves similarly to the centre of a group (see for instance \cite{Bonatto,MR4256133,JeVAV22x}). Indeed, if $g$ is any element of $\Ann(x)$, then $b+g=g+b=g\circ b=b\circ g$, for all $b\in B$, so that $b\ast g=0$ and $g\in\Fix^l(x)$; thus, \hbox{$\Ann(B)=\bigcap_{x\in B}\Ann(x)$.}

\begin{defn}\label{deffc}
An element $x\in B$ has the {\it property \ps} if 
$\big|(B,+):\operatorname{Fix}^r(x)\cap C_x^+\big|$ and $\big|(B,\circ):\operatorname{Fix}^l(x)\cap C_x^\circ\big|$ are finite.  A skew brace $B$ has the {\it property \PS} if all its elements have property \ps. 
\end{defn}

We chose to denote our property by ‘‘(S)’’ to prevent confusion between the group theoretical concept of $FC$-group and skew braces with the property (S). By doing so, a quick glance enables readers to distinguish whether we are focusing on the attributes of $FC$-groups or (S)-braces. 

Special cases of skew braces with property \PS are those skew braces in which~$B^{(2)}$ and $[B,B]_+$ are finite, since then also $[B,B]_\circ$ is finite. Every skew brace $B$ such that $|B/\Ann(B)|<\infty$ has property \PS by \cite[Theorem 5.4]{MR4256133}. Other natural examples of skew braces with property \PS can be constructed by considering the restricted direct product of (possibly infinitely many) finite skew braces. It is also straightforward to show that the class of skew braces with property \PS is closed with respect to forming sub-skew braces, homomorphic images and restricted direct products. On the other hand, not every skew brace has property \PS. In fact, if $B$ is any finite skew brace such that $\operatorname{Ann}(B)<B$ (such as, for instance, the skew brace in Example~\ref{annihilatornontrivial}), then the unrestricted direct product $C=\operatorname{Cr}_{i\in\mathbb{N}}B_i$ of infinitely many copies of $B$ (that is, $B_i\simeq B$, for every $i\in\mathbb{N}$) does not have property \PS: if $x\in B\setminus\operatorname{Ann}(B)$, then $(x)_{i\in\mathbb{N}}\in C$ has infinitely many conjugates. 

\begin{exa}\label{annihilatornontrivial}
    Let $B$ be the brace whose additive group is isomorphic to $\mathbb{Z}/4\mathbb{Z}$ and whose multiplication is defined by $x_1\circ x_2 = x_1+x_2+2x_1x_2$ (see \cite[Proposition 2.4]{Ba18}). Then the annihilator of $B$ is strictly contained in $B$, and is given by $\Ann(B)=\left\{0, 2\right\}$.
\end{exa}

\begin{notation}
    Let $B$ be a skew brace. If $x\in B$, any element of type $g\ast x$, $x\ast g$, $g\circ x\circ g^{-1}$, $g+x-g$, for some $g\in B$, will be referred to as a {\it conjugate} of $x$. 
\end{notation}

This terminology is justified by the fact that the operation $\ast$ corresponds to commutation in the semidirect product \hbox{$(B,+)\rtimes_\lambda(B,\circ)$}. Under this terminology, an \ps-element of a skew brace is just an element having finitely many conjugates.

It is a straightforward (but long) exercise to show that in a two-sided brace $B$, the set of all \ps-elements is an ideal. We state this fact as follows, only sketching the proof. 

\begin{pro}
Let $B$ be a two-sided brace. Then the set $S$ of all \ps-elements of~$B$ is an ideal of $B$.
\end{pro}
\begin{proof}
Since $B$ is a brace, we need to show that $S$ is a subgroup of $(B,+)$, a normal subgroup of $(B,\circ)$ and that $S$ is $\lambda$-invariant. We show that $S\leq (B,+)$, since the other statements can similarly be proved.

Let $b_1,b_2\in S$. We need to show that $b_1+b_2\in S$. Since $$(b_1-b_2)^{c,+}=b_1^{c,+}-b_2^{c,+},$$ for all $c\in B$, we have that the set $\big\{(b_1-b_2)^{c,+}\,:\, c\in B\big\}$ is finite. Now, note that for all $c\in B$ we have $$c^{-1}=c^{-1}\circ (c-c)=-c^{-1}+c^{-1}\circ(-c)\implies 2c^{-1}=c^{-1}\circ (-c)$$ and similarly $$2c^{-1}=(-c)\circ c^{-1}.$$ This implies that $$0=c\circ(b_1-b_1)\circ c^{-1}=b_1^{c,\circ}+(-b_1)^{c,\circ}$$ and hence that the set $\big\{(-b_1)^{c,\circ}\,:\, c\in B\}$ is finite. Therefore $$(b_1-b_2)^{c,\circ}=b_1^{c,\circ}+(-b_2)^{c,\circ}$$ assumes finitely many values, as $c$ ranges in $B$. Finally, the fact that the sets $$\{c\ast (b_1-b_2)\,:\, c\in B\}\quad\textnormal{and}\quad\{(b_1-b_2)\ast c\,:\, c\in B\}$$ are finite depend upon the distributive laws. Thus, $b_1-b_2\in S$, and we are done.
\end{proof}

However, we are not able to establish if this is the case also for skew braces.

\medskip

Let $B$ be a skew brace with property \PS. Put $G=N\rtimes X$, where $N=(B,+)$, $X=(B,\circ)$, and  the action of $X$ on $N$ is given by $\lambda$. A direct consequence of the definitions is that both $X$ and $N$ are $FC$-groups, but $G$ is also an $FC$-group: to see this, note that Equation \eqref{equationfinale} yields that both $X$ and $N$ are contained in the $FC$-centre of $G$, so $G=FC(G)$. This fact will be very useful in our considerations below and will sometimes be employed without any further notice.

\medskip

Our first main result deals with finitely generated skew braces with property \PS, but first, we need a lemma and some definitions.

\begin{defn}[see also \cite{MR4256133}, p.6]
A skew brace $B$ is {\it finitely generated} if there is a finite subset $E$ of~$B$ such that $B$ is the smallest sub-skew brace of $B$ containing~$B$.
\end{defn}

\begin{lem}\label{lemma1}
Let $B$ be a skew brace whose additive group is generated by the \ps-elements $x_1,\ldots,x_n$. Then $B/\Ann(B)$ is finite.
\end{lem}
\begin{proof}
Consider the map $$f:\, B/\Ann(B)\rightarrow \big(B^{(2)}\big)^{2n}\times \big([B,B]_+\big)^n$$ given by the assignment $$f\left(b+\operatorname{Ann}(B)\right)=\left(b\ast x_1,\ldots,b\ast x_n,x_1\ast b,\ldots,x_n\ast b,[b,x_1]_+,\ldots,[b,x_n]_+\right).$$ It has been shown in the second half of \cite[Theorem 5.4]{MR4256133} that the function $f$ is well-defined and injective. Since every $x_i$ is an \ps-element of $B$, it follows that the image of $f$ is a finite set, so also $B/\Ann(B)$ is finite, and the statement is proved. 
\end{proof}

Let $\mathcal{S}=\{x_1,\ldots,x_n\}$ be symbols. A {\it $b$-word} with respect to $\mathcal{S}$ is a sequence of symbols recursively defined as follows: the empty sequence is a $b$-word and such are the $1$-element sequences $x_1$, $x_2$, \ldots, $x_n$; if we have two $b$-words $w_1$ and $w_2$, then the sequences $w_1\circ w_2$, $w_1+w_2$, $-w_1$, $w_1^{-1}$ are $b$-words.  We define the {\it weight} $\mathfrak{w}(w)$ of a $b$-word $w$ with respect to $\mathcal{S}$ as follows. Let $\mathfrak{w}(\emptyset)=0$ and $\mathfrak{w}(x_i)=1$ for each $i=1,\ldots,n$. If $w_1$ and $w_2$ are $b$-words whose weights are $n_1$ and $n_2$, respectively, then the $b$-words $w_1\circ w_2$ and $w_1+w_2$ both have weight $n_1+n_2$, while the $b$-words $-w_1$ and $w_1^{-1}$ both have weight $n_1+1$.

Now, let $B$ be a skew brace and let $b_1,\ldots, b_n$ be elements of $B$. It is clear that if $w(x_1,\ldots,x_n)$ is any $b$-word, then we may evaluate $w(b_1,\ldots,b_n)$ in $B$. Thus, the smallest sub-skew brace $C$ generated by $b_1,\ldots, b_n$ in $B$ is precisely the set of all evaluations of $b$-words. Moreover, if $c$ is any element of $C$, the {\it weight}~$\mathfrak{w}(c)$ of $c$ (with respect to $b_1,\ldots,b_n$) is defined as the smallest weight of a $b$-word $w(x_1,\ldots,x_n)$ such that $c=w(b_1,\ldots,b_n)$.

\begin{thm}\label{theoremfg}
Let $B$ be a skew brace with property \PS. The following statements are equivalent:
\begin{itemize}
\item[\textnormal{(1)}] $B$ is finitely generated.
\item[\textnormal{(2)}] $(B,+)$ is finitely generated.
\item[\textnormal{(3)}] $(B,\circ)$ is finitely generated.
\end{itemize}
\end{thm}
\begin{proof}
It is clear that (3) $\implies$ (1) and (2) $\implies$ (1). Assume (1) and let $b_1,\ldots,b_n$ be generators of $B$ containing their additive inverses. Let $G=N\rtimes X$, where the action of $X=(B,\circ)$ on $N=(B,+)$ is given by $\lambda$. As we already mentioned, $G$ is an $FC$-group, so in particular, $[G,G]$ is locally finite.

It follows from Equation \eqref{equationfinale} that, for each $c\in B$ and $i=1,\ldots,n$, the element $(c\ast b_i,0)$ of $N$ is contained in $[G,G]$. Clearly, also the elements $\big([c,b_i]_+,0\big)$ are contained in $[G,G]\cap N$. Now, since $G$ is an $FC$-group, the elements of types $(c\ast b_i,0)$ and $([c,b_i]_+,0)$ are periodic and finite in number; let $E$ be their set. It follows from Dietzmann's lemma (see the preliminaries)  that the normal closure $H$ of $E$ in $G$ is finite (and is certainly contained in~$N$). 

Let $g_1=(b_1,0),\ldots,g_n=(b_n,0)$. We claim that $N=U$, where $$U=\big\langle g_1,\ldots, g_n,H\big\rangle,$$ is the subgroup generated by $g_1,\ldots, g_n$ and $H$. To this aim, by induction on the weight $\mathfrak{w}(c)$ of $c$ (with respect to $b_1,\ldots,b_n$), we prove that~an arbitrary element $g=(c,0)$ of $N$ belongs to $U$. If $\mathfrak{w}(c)= 0$, this is obvious. Assume $\mathfrak{w}(c)\geq1$ and the usual induction hypothesis. By definition, there are elements $c_1$ and $c_2$ of~$B$ such that one of the following alternatives holds:

\begin{itemize}
\item[(i)]  $c=c_1\circ c_2$ and $\mathfrak{w}(c_1),\mathfrak{w}(c_2)<\mathfrak{w}(c)$.
\item[(ii)] $c=c_1+c_2$ and $\mathfrak{w}(c_1),\mathfrak{w}(c_2)<\mathfrak{w}(c)$. 
\item[(iii)] $c=-c_1$ and $\mathfrak{w}(c)=\mathfrak{w}(c_1)+1$.
\item[(iv)] $c=c_1^{-1}$ and $\mathfrak{w}(c)=\mathfrak{w}(c_1)+1$.
\end{itemize}

It is clear that cases (ii) and (iii) can be neglected, so we only need to deal with cases (i) and (iv). Let us first deal with (i). In this case, $c=c_1\circ c_2$ and $\mathfrak{w}(c_1),\mathfrak{w}(c_2)<\mathfrak{w}(c)$. By induction, we can write $(c_2,0)+H$ as a sum (with repetitions) of the elements $g_1+H,\ldots,g_n+H$, so, by the skew distributive law, we may assume $c_2=b_i$, for some $i$: notice here that for example $$c_1\circ (b_1+h)=c_1\circ b_1-c_1+c_1\circ h=c_1\circ b_1+c_1\ast h+h,$$ whenever $(h,0)\in H$, and $(c_1\ast h+h,0)\in H$.  However, $c_1\ast b_i$ is contained in $H$, so $$g+H=(c_1\circ b_i,0)+H=(c_1,0)+(b_i,0)+H$$ and again $g$ belongs to $U$, by induction. 

Now we deal with (iv). In this case, $c=c_1^{-1}$ and $\mathfrak{w}(c)=\mathfrak{w}(c_1)+1$, so by induction $c_1$ belongs to $U$. For simplicity's sake, we write every element $(b,0)$ of $N$ as $b$; under this terminology, $a\ast b$ means $(a\ast b,0)$.

We claim that $a^{-1}+H=-a+H$ every time $a$ belongs to $U$; this completes the proof in this case. Write $a=b_1'+\ldots+b_m'+h$ for some elements $b_1',\ldots,b_m'\in\{b_1,\ldots,b_n\}$ and $h\in H$. Clearly, for each $i$, $$H=b_i\circ b_i^{-1}+H=b_i+b_i^{-1}+H,$$ so $-b_i+H=b_i^{-1}+H$. Now, since $H$ is a normal subgroup of $G$, it contains all elements of the form $a\ast h$, where $a\in N$ and $h\in H$, and consequently {\small 
\begin{align*}
(b_1'+&\ldots+b_m'+h)\circ(-b_1'+\ldots-b_m'-h)+H\\
&=(b_1'+\ldots+b_m'+h)\circ(-b_1'+\ldots-b_m')-b_1'+\ldots-b_m'-h\\
&\quad+(b_1'+\ldots+b_m'+h)\circ(-h)+H\\
&=(b_1'+\ldots+b_m'+h)\circ(-b_1'+\ldots-b_{m}')-b_1'+\ldots-b_m'-h\\
&\quad+b_1'+\ldots+b_m'+h-h+H\\
&=(b_1'+\ldots+b_m'+h)\circ(-b_1'+\ldots-b_{m}')-h+H\\
&=\ldots=(b_1'+\ldots+b_m'+h)\circ(-b_1')-b_2'+\ldots-b_m'-h+H\\
&=b_1'+\ldots+b_m'+h-b_1'-b_2'-\ldots-b_m'-h+H=H
\end{align*} }

Therefore there is an element $d$ of $H$ such that $$(b_1'+\ldots+b_m'+h)\circ(-b_1'+\ldots-b_m'-h)=d.$$ However, $H$ is a normal subgroup of $G$, so $H$ is a group with respect to $\circ$, which means that $$d^{-1}\circ (b_1'+\ldots+b_m'+h)=(-b_1'+\ldots-b_m'-h)^{-1}.$$ Repeated application of the skew distributive law shows that $$(-b_1'+\ldots-b_m'-h)^{-1}+H=d^{-1}\circ (b_1'+\ldots+b_m'+h)+H=(b_1'+\ldots+b_m')+H.$$ The fact that $b_1,\ldots,b_m,m,h$ were arbitrary proves the claim.

\medskip

Finally, since $H$ is a finite subgroup of $(B,+)$, it follows that $(B,+)$ is finitely generated. Moreover, Lemma \ref{lemma1} yields that $B/\Ann(B)$ is finite. Since $(B,+)$ is finitely generated also $\Ann(B)$ is finitely generated as a subgroup of $(B,+)$, and this means that it is also finitely generated as a subgroup of $(B,\circ)$. The fact that $\Ann(B)$ has finite index in $(B,\circ)$ means that $(B,\circ)$ is finitely generated and completes the proof of the theorem.
\end{proof}

\begin{cor}
Let $B$ be a skew brace with property \PS. If $B$ is finitely generated, then $B/\operatorname{Ann}(B)$ is finite.
\end{cor}
\begin{proof}
This is a consequence of Theorem \ref{theoremfg} and Lemma \ref{lemma1}.
\end{proof}

\medskip

The set of all periodic elements of a group $G$ is denoted by $T(G)$. For a skew brace $(B,+,\circ)$, we denote by $T_+(B)$ the set of periodic elements of the additive group $(B,+)$ and by $T_{\circ}(B)$ the set of periodic elements of the multiplicative group $(B,\circ)$. We say that a skew brace $B$ with property \PS  is {\it periodic} if $(B,+)=T_+(B)$, and is {\it torsion-free} if $T_+(B)=\{0\}$.

\begin{cor}
Let $B$ be a periodic skew brace. If $B$ is finitely generated, then~$B$ is finite.
\end{cor}

\medskip

Let $B$ be a skew brace. We say that $B$ is {\it locally finite} if every finitely generated sub-skew brace of $B$ is finite.

\begin{cor}\label{periodicimplieslocallyfinite}
Let $B$ be a periodic skew brace. Then $B$ is locally finite.
\end{cor}

\begin{cor}
Let $B$ be a skew brace with property \PS. If $B$ is finitely generated, then $B/\operatorname{Ann}(B)$ is finite.
\end{cor}
\begin{proof}
This is a consequence of Theorem \ref{theoremfg} and Lemma \ref{lemma1}.
\end{proof}

\medskip

If $G$ is an $FC$-group, then $T(G)$ is a normal subgroup of $G$. If $B$ is a skew brace with property \PS, then both $(B,+)$ and $(B,\circ)$ are $FC$-groups, so $T_+(B)$ and $T_\circ(B)$ are normal subgroups of $(B,+)$ and $(B,\circ)$, respectively. Our next result justifies the definition of a periodic skew brace we gave above.

\begin{thm}\label{theoperiodicfc}
Let $B$ be a skew brace with property \PS. The following statements hold:
\begin{itemize}
    \item[\textnormal{(1)}] $T_+(B)=T_\circ(B)$.
    \item[\textnormal{(2)}] $T_+(B)$ is an ideal of $B$.
    \item[\textnormal{(3)}] $T_+\big(B/T_+(B)\big)=\{0\}$. 
    \item[\textnormal{(4)}] If $B$ is finitely generated, then $T_+(B)$ is finite.
\end{itemize}
\end{thm}
\begin{proof}
Let $G=N\rtimes X$, where the action of $X=(B,\circ)$ on $N=(B,+)$ is given by~$\lambda$. 

\medskip

\noindent(1)\quad Let $a\in B$. Of course, $a$ can be regarded as an element $g_1=(a,0)$ of~$N$ and as an element $g_2=(0,a)$ of $X$. We need to show that $g_1$ is periodic if and only if~$g_2$ is periodic.

Assume first that $g_1$ is periodic. Since $G$ is an $FC$-group, Dietzmann's lemma yields that the normal closure $M=\langle g_1\rangle^G$ of $\langle g_1\rangle$ in $G$ is finite. Then $M$ contains $[M,X]$, which means that if $M$ contains elements of type $(b,0)$ and $(c,0)$ for some $b,c\in B$, then by Equation~\eqref{equationfinale} it also contains $(b\ast c,0)$ and, consequently, $(b\circ c,0)$. Therefore, every element $(b,0)\in M$ is such that $b$ is a periodic element of $X$. In particular, $g_2$ is periodic.

Conversely, assume $g_2$ is periodic. Since $G$ is an $FC$-group, its derived subgroup~$[G,G]$ is locally finite. In particular, $[X,N]$ is a locally finite normal subgroup of~$G$. On the other hand, Equation~\eqref{equationfinale} shows that $[X,N]$ contains every element of type $(b,0)$, where $b\in B^{(2)}$, so $$(b_1+b_2,0)+[X,N]=(b_1\circ b_2,0)+[X,N],$$ for every $b_1,b_2\in B$. It follows that there is a positive integer $m$ such that $$(c,0)=(m a,0)\in[X,N].$$ However, this means that $c$ is a periodic element of $N$, and so $a$ is a periodic element of $N$. Thus $g_1$ is periodic, and the statement is proved.

\medskip

\noindent(2)  As a consequence of the previous point $T_+(B)$ is a normal subgroup of both~$N$ and $X$. It only remains to show that $T_+(B)$ is $\lambda$-invariant. However, this immediately follows from the observation that the set of all periodic elements of $N$ is a characteristic subgroup of $N$, so it is normal in $G$. Therefore $T_+(B)$ is $\lambda$-invariant and hence is an ideal of $B$.

\medskip

\noindent(3) This is an immediate consequence of (2).

\medskip

\noindent(4) It follows from Theorem~\ref{theoremfg} that $(B,+)$ is finitely generated. Since $(B,+)$ is an $FC$-group, we have that $T_+(B)$ is finite which proves the statement.
\end{proof}

\medskip

 Periodic skew braces play a very relevant role in the structure of an arbitrary skew brace with property \PS as the following results show.

\begin{thm}\label{thgzglf}
Let $B$ be a skew brace with property \PS. Then $B/\Ann(B)$ is locally finite.
\end{thm}
\begin{proof}
By Corollary \ref{periodicimplieslocallyfinite}, we only need to prove that $B/\Ann(B)$ is periodic.

Let $b\in B$. We need to show that there is a positive integer $n$ such that $nb$ belongs to $\operatorname{Ann}(B)$. By Theorem \ref{theoperiodicfc}, this is equivalent to finding a positive integer $n$ such that $b^n$ is contained in $\operatorname{Ann}(B)$.

Let $G=N\rtimes_\lambda X$, where $N=(B,+)$ and $X=(B,\circ)$. Since $G$ is an $FC$-group, we have that $G/Z(G)$ is a locally finite group. Let $x=(0,b)$. Then there is a positive integer $m$ such that $$(0,b^m)=x^m\in Z(G).$$ It follows from Equation \eqref{equationfinale} that $b^m\ast c=0$, for every $c\in B$; consequently, $b^m\circ c=b^m+c$ for every $c\in B$. So far, we have that $b^m$ belongs to $\Ker(\lambda)\cap Z(B,\circ)$. Let $g=(b^m,0)$. Then there is a positive integer $\ell$ such that $$(b^{\ell m},0)=(\ell b^m,0)=g^\ell\in Z(G).$$ Thus $b'=b^{\ell m}$ belongs to $Z(B,+)$. But $\Ker(\lambda)\cap Z(B,\circ)$ is a subgroup of $(B,\circ)$, so $b'=(b^m)^\ell$ still belongs to $\Ker(\lambda)\cap Z(B,\circ)$ and the statement is proved.
\end{proof}

\medskip

Let $n\in\mathbb{N}$. A skew brace is locally finite {\it of finite exponent dividing $n$} if every finitely generated sub-skew brace is finite and both its additive and multiplicative groups have exponent dividing $n$. In combination with Theorem \ref{thgzglf}, our next result has some relevant consequences in the theory of skew braces with property \PS. It generalises well-known results of Schur and Mann (see \cite{Mann}) in group theory, and it extends \cite[Theorem 5.4]{MR4256133}.

\begin{thm}\label{schurlf}
Let $B$ be a skew brace.

\begin{itemize}
    \item[\textnormal{(1)}] If $B/\Ann(B)$ is locally finite, then~$B^{(2)}$ is locally finite.
    \item[\textnormal{(2)}] If $B/\Ann(B)$ is locally finite  of finite exponent dividing $n$ \textnormal{(}resp., finite of exponent dividing $n$\textnormal{)}, then~$B^{(2)}$ is locally finite of finite exponent dividing $f(n)$ \textnormal{(}resp., finite of exponent dividing $f(n)$\textnormal{)}, where $f(n)$ depends only on~$n$.
\end{itemize}
\end{thm}
\begin{proof}
(1)\quad Let $b_1,\ldots,b_n\in B^{(2)}$ and let $A$ be the sub-skew brace generated by $b_1,\ldots,b_n$. We need to show that $A$ is finite. By definition of $B^{(2)}$, we can find elements $$c_{1,1},\ldots,c_{n,1},\ldots,c_{1,n},\ldots,c_{n,n}\quad\textnormal{and}\quad d_{1,1},\ldots,d_{n,1},\ldots,d_{1,n},\ldots,d_{n,n}$$ such that $$b_1\in\langle c_{1,1}\ast d_{1,1},\ldots,c_{n,1}\ast d_{n,1}\rangle_+,\ldots,b_n\in\langle c_{1,n}\ast d_{1,n},\ldots,c_{n,n}\ast d_{n,n}\rangle_+.$$ Now, consider the sub-skew brace $C/\Ann(B)$ generated by $$c_{1,1}+\Ann(B),\ldots,c_{n,n}+\Ann(B),d_{1,1}+\Ann(B),\ldots,d_{n,n}+\Ann(B).$$ By the hypothesis, $C/\Ann(B)$ is finite. But clearly, $\Ann(B)$ is contained in $\Ann(C)$, so Lemma~\ref{lemma1} yields that~$C^{(2)}$ is finite. Since $A$ is a sub-skew brace of $C^{(2)}$, it follows that $A$ is finite and the statement is proved.

\medskip

\noindent(2)\quad Assume first that $B/\Ann(B)$ is finite and both its additive and multiplicative groups have exponent dividing $n$. Let $N=(B,+)$, $X=(B,\circ)/\Ker(\lambda)$ and $G=N\rtimes_\lambda X$. Since the exponent of $(B,\circ)/\Ann(B)$ divides $n$, it follows that the exponent of $X$ also divides~$n$. 

It is easy to see that $\Ann(B)$ as a subgroup of $N$ is contained in the centre $Z(G)$ of $G$, so $G/Z(G)$ is finite of exponent dividing $n^2$. It follows from \cite{Mann} that $[G,G]$ is finite of exponent dividing $g(n^2)$, where $g$ is a function depending only on $n$. Finally, Equation \eqref{equationfinale} also shows that $B\ast B$ is finite of exponent dividing $g(n^2)$. The statement is proved in this case.

In order to deal with the locally finite case, we just need to repeat the proof of~(1), replacing Lemma \ref{lemma1} by the first part of (2).
\end{proof}

\medskip

As a consequence of Theorems  \ref{thgzglf} and \ref{schurlf} we have the following results.

\begin{cor}\label{corschurlf}
Let $B$ be a skew brace with property \PS. Then $B^{(2)}$ is locally finite.
\end{cor}

\begin{cor}\label{idealtorsionfreeskewlefttrivial}
Let $B$ be a skew brace with property \PS. If $I$ is a torsion-free ideal of $B$, then $I\leq\Ann(B)$.
\end{cor}
\begin{proof}
By assumptions, $T_+(I)=\{0\}$, so Corollary~\ref{corschurlf} yields \hbox{$I\ast B=B\ast I=\{0\}$.} On other hand, any $FC$-group has a locally finite commutator subgroup and consequently $[I,B]_\circ,[I,B]_+\leq T_+(I)={0}$. The statement is proved.
\end{proof}

\begin{cor}\label{torsionfreeskewlefttrivial}
Let $B$ be a torsion-free skew brace. Then $B$ is a trivial brace.
\end{cor}

\medskip

Theorem~\ref{thgzglf} also has an immediate consequence on the induced solution $(B,r_B)$ of a skew brace $B$ with property \PS.

\begin{cor}
Let $B$ be a skew brace with property \PS. Then the orbits of $r_B$ as an element of $\operatorname{Sym}(B\times B)$ are finite.
\end{cor}
\begin{proof}
Let $F$ be any subset of $B$. It follows from Lemma~\ref{lemma1} that the sub-skew brace $C$ generated by $F$ is finite over $\Ann(C)$. Let $n$ be the index of $\Ann(C)$ in~$C$. Then \cite[Lemma 4.12 (1)]{Vendramin} shows that $r_C^{2n}(a,b)=(a,b)$, for all $a,b\in C$ and completes the proof of the statement. 
\end{proof}

\medskip

Let $(X,r)$ be a non-degenerate solution of the YBE and $\iota$  the natural map $\iota:\, X\longrightarrow G(X,r)$. If $\iota$ is injective, we say that $(X,r)$ is \emph{injective}. On the other hand, it is possible for $\iota$ to be non-injective. 
Take for instance $X=\{1,2,3,4\}$, $f=(1\ 2)$ and $g=(3\ 4)$;  the map $r:X\times X \to X\times X$ given by $r(x,y)=(f(y),g(x))$ is a non-degenerate solution of the YBE but it is not injective since $1=2$ and $3=4$ in $G(X,r)$.
In this case, the injectivization of $(X,r)$ can be useful (see \cite{Soloviev}). If $\tilde{X}=\iota(X)$ and $\tilde{r}:\, \tilde{X}\times\tilde{X}\longrightarrow\tilde{X}\times\tilde{X}$ is the restriction of $r_{G(X,r)}$ to $\tilde{X}\times\tilde{X}$, then the pair $(\tilde{X},\tilde{r})$ is an injective non-degenerate solution of the YBE called the {\it injectivization} of $(X,r)$ and is such that $G(X,r)\simeq G(\tilde{X},\tilde{r})$. Recall also that $(X,r)$ is a {\it trivial} solution if $r(x,y)=(y,x)$ for all $x,y\in X$. Note here that the previous example shows that the injectivization of a non-involutive and non-degenerate solution may be involutive (even trivial).

The trivial solution is involutive, and the following result shows that if $G(X,r)$ has property \PS, then the converse holds for the injectivization of $(X,r)$.

\begin{cor}\label{corinvolutive}
Let $(X,r)$ be a non-degenerate finite solution of the YBE such that $G(X,r)$ has property \PS. Then $(\tilde{X},\tilde{r})$ is involutive if and only if $(\tilde{X},\tilde{r})$ is trivial.
\end{cor}
\begin{proof}
If $(\tilde{X},\tilde{r})$ is involutive, then $G(X,r)$ is torsion-free (see \cite[Theorem 6.5]{MR4256133}) and so is a trivial brace by Corollary~\ref{torsionfreeskewlefttrivial}. Then $(\tilde{X},\tilde{r})$ is the trivial solution. 
\end{proof}

\begin{thm}
Let $B$ be a skew brace. If $I$ is a torsion-free sub-brace of~$\Ann(B)$ such that $B/I$ is a periodic skew brace, then $B$ has property \PS.
\end{thm}
\begin{proof}
Let $b\in B$. We need to show that the indices $$\big|(B,+):\operatorname{Fix}^r(b)\cap C_b^+\big|\quad\textnormal{ and }\quad\big|(B,\circ):\operatorname{Fix}^l(b)\cap C_b^\circ\big|$$ are finite. Since $B/I$ has property \PS, we have that  $$\big|(B/I,\circ):\operatorname{Fix}^l_{B/I}(b+I)\big|<\infty.$$ Let $U/I=\operatorname{Fix}^l_{B/I}(b+I)$. Then $U$ is a subgroup of finite index of $(B,\circ)$. Moreover, if $c\in U$, then $c\ast b$ belongs to $I$. But $B^{(2)}$ is locally finite by Corollary \ref{corschurlf}, and so~\hbox{$c\ast b=0$.} It follows that $\operatorname{Fix}^l_B(b)$ has finite index in $(B,\circ)$. Similarly, we prove that $\operatorname{Fix}^r_B(b)$ has finite index in $(B,+)$.

A theorem of Chernikov \cite{Cernikov} (which is in fact the group theoretical analogue of the result we are proving) now shows that also the indices $$|(B,+):C_b^+|\quad\textnormal{and}\quad|(B,\circ):C_b^\circ|$$ are finite, and the proof of the statement follows.
\end{proof}

\medskip

One of the most interesting connections between arbitrary and periodic skew braces is given by the following result.

\begin{thm}\label{periodicreduction}
Let $B$ be a skew brace with property \PS. Then $B$ is a subdirect product of a trivial brace $C$ and a periodic skew  brace $D$. Moreover, $C$ and $D$ are homomorphic images of $B$.
\end{thm}
\begin{proof}
Let $T=T_+(B)$. By Corollary \ref{torsionfreeskewlefttrivial},  $B/T$ is a trivial brace. Now, every subgroup of the additive group of $\Ann(B)$ is an ideal of $B$. Let $F$ be a maximal free abelian subgroup of $\Ann(B)$ (this certainly exists by Zorn's lemma). Then~$F$ is an ideal of $B$ and $\Ann(B)/F$ is periodic. Thus $B/F$ is also periodic. 

Since $F\cap T=\{0\}$, we have that the map $$b\in B\mapsto \big(b+T,\, b+F)\in B/T\times B/F$$ is a monomorphism of skew  braces and the statement follows.
\end{proof}

\medskip

%
%

\medskip

In dealing with periodic skew braces, one would hope an analogue of Dietzmann's lemma for $FC$-groups would hold. Unfortunately, it is not quite clear if a result of this type holds or not.

\begin{question}\label{question1}
Let $B$ be a periodic skew brace. Is it true that any finitely generated ideal of $B$ is finite?
\end{question}

A positive answer to the above question would provide a great deal of information about the structure of an arbitrary skew brace with property \PS. In this context, the best we can do is Theorem~\ref{extension}. Here, an ideal $I$ of a skew brace $B$ with property \PS is said to be {\it good} if every finitely generated ideal $J$ of $B$ with $J\subseteq I$ is finitely generated as a skew brace. Using this terminology, Question~\ref{question1} can be rephrased as follows: is every periodic skew brace good? Actually, as a consequence of Theorems \ref{thgzglf}, \ref{extension} and \ref{goodsocle}, it turns out that the previous question is equivalent to the following one: is every skew brace with property \PS good?

It is clear that every restricted direct product of finite skew braces is a good skew brace with property \PS. We anticipate that, in Section~\ref{bfc}, we will show that this is also the case for the class of skew braces with property \PBS (see Definition~\ref{defbfc}); in particular, this will be a direct consequence of Corollary~\ref{corextension} and The\-o\-rem~\ref{thm1}.

\begin{thm}\label{extension}
Let $B$ be a skew brace with property \PS. If $I\subseteq J$ are ideals of $B$ such that $I$ and $J/I$ are good ideals of $B$ and $B/I$, respectively, then $J$ is a good ideal of $B$.
\end{thm}
\begin{proof}
Let $F$ be a finite subset of $J$ (containing its additive inverses). We need to show that the ideal $L$ generated by $F$ in $B$ is finitely generated as a skew brace. 

Note that a combination of Theorem~\ref{theoperiodicfc} and Theorem~\ref{schurlf} shows that a finitely generated skew brace with property \PS satisfies the ascending chain condition on sub-skew braces (recall that a finitely generated $FC$-group has a finite commutator subgroup), so to achieve our goal, we just need to prove that $L$ is contained in some finitely generated sub-skew brace of $B$; this remark also allows us to factor out every ideal which is finitely generated as a skew brace.

Now, the ideal $K/I$ generated by the image of $F$ in $B/I$ is finitely generated as a skew brace, so it follows from Theorem~\ref{theoperiodicfc} that $T_+(K/I)$ is finite; the same result also shows that $H/I=T_+(K/I)$ is an ideal of $B/I$. Moreover, Theorem~\ref{schurlf} shows that $K/H$ is a finitely generated torsion-free trivial brace. It is therefore enough to split the proof in two cases according to $K/I$ being finite or a finitely generated torsion-free trivial brace.

Let us first deal with the latter case, \hbox{i.e.} $K/I$ is a torsion-free trivial brace generated by $x_1,\ldots,x_n$. Let $A$ be the sub-skew brace of $B$ generated by $F$ and $x_1,\ldots,x_n$. Then $A$ satisfies the ascending chain condition on sub-skew braces, and consequently, $A\cap I$ is a finitely generated skew brace. Since~$I$ is good, the ideal $M$ generated by~\hbox{$A\cap I$} is finitely generated as a skew brace and, factorising by $M$, we may assume $A\cap I=\{0\}$. In particular, $A$ is a trivial brace generated by $x_1,\ldots,x_n$. Since $B$ has property \PS, there are finitely many possibilities for $x_i^{\circ,g}=g\circ x_i\circ g^{-1}$, where $g\in B$ and $i=1,\ldots,n$, and, by the hypothesis on $K/I$, each of these possibilities can be written as a linear combination $u(i,g)=m_1x_1+\ldots+m_nx_n$ modulo $I$. Similarly, there are finitely many possibilities for the elements of type $x_i^{+,g}=g+x_i-g$ (resp., $g\ast x_i$), where $g\in B$ and $i=1,\ldots,n$, and each of these can be written as a linear combination $v(i,g)$ (resp., $z(i,g)$) modulo $I$ of $x_1,\ldots,x_n$. Thus the set $$U=\big\{u(i,g)-x_i^{\circ,g},\; v(i,g)-x_i^{+,g},\; z(i,g)-g\ast x_i:\, g\in B,\, i=1,\ldots,n\big\}$$ is finite and is contained in $I$, so the ideal generated by $U$ is finitely generated as a skew brace and then we may factor it out. Now, since $A$ is generated additively and multiplicatively by $x_1,\ldots,x_n$, we have that $A$ is a normal subgroup of both $(B,+)$ and $(B,\circ)$. Moreover, $g\ast x$ belongs to $A$, for any $g\in B$ and $x\in A$, since $g\ast x_i\in A$, for every $i$ and we may write $x=v_1x_1+\ldots+v_nx_n$ for some integers $v_1,\ldots,v_n$. It follows that $A$ is an ideal which is finitely generated as a skew brace and the statement is proved in this case.

We now turn to the case in which $K/I$ is finite. This case can essentially be dealt with as above when we replace $x_1,\ldots,x_n$ by a complete set of representatives of $K$ modulo $I$.
\end{proof}

\begin{cor}\label{corextension}
Let $B$ be a skew brace with property \PS. If $$\{0\}=I_0\subseteq I_1\subseteq \ldots I_\alpha\subseteq I_{\alpha+1}\subseteq\ldots\; J=\bigcup_{\alpha<\mu} I_\alpha$$ is an ascending chain of ideals whose factors are good, then $J$ is a good ideal of $B$.
\end{cor}
\begin{proof}
Since $J$ is an ideal of $B$, being a union of a chain of ideals, we only need to show that $J$ is good. Let $\gamma$ be the smallest ordinal number such that $I_\gamma$ is not good. Of course, $\gamma>0$. Moreover, by Theorem~\ref{extension}, $\gamma$ cannot be a successor, so $\gamma$ is the limit. We have reached a contradiction since the union of any chain of good ideals is good.
\end{proof}

\medskip

The above corollary shows that skew braces with property \PS contain many natural ideals which are good. 

Let $B$ be a skew brace. We recursively define the {\it upper socle series} of~$B$ as follows. Put $\operatorname{Soc}_0(B)=\{0\}$ and $\operatorname{Soc}_1(B)=\operatorname{Soc}(B)=\operatorname{Ker}(\lambda)\cap Z(B,+)$. If~$\alpha$ is an ordinal number, put $$\operatorname{Soc}_{\alpha+1}(B)/\operatorname{Soc}_\alpha(B)=\operatorname{Soc}\big(B/\Soc_{\alpha}(B)\big).$$ If $\mu$ is a limit ordinal, put $$\operatorname{Soc}_\mu(B)=\bigcup_{\gamma<\mu}\operatorname{Soc}_\beta(B).$$ The smallest ordinal number $\alpha$ such that $\operatorname{Soc}_\alpha(B)=\operatorname{Soc}_{\alpha+1}(B)$ is the {\it upper socle length} of $B$ and is denoted by $s(B)$; the last term of the upper socle series is the {\it hyper-socle} of $B$ and is denoted by $\overline{\operatorname{Soc}}(B)$. Note that every term of the upper socle series is an ideal of $B$ and it is easy to see that $\overline{\operatorname{Soc}}\big(B/\overline{\operatorname{Soc}}(B)\big)$ is zero. This series of ideals is relevant in some nilpotency theory of skew braces: it turns out, for instance, that a skew brace $B$ of nilpotent type is right nilpotent if and only if $B=\operatorname{Soc}_n(B)$ for some positive integer $n$ (see \cite[Lemmas 2.15 and 2.16]{Cedo}).

\begin{thm}\label{goodsocle}
Let $B$ be a skew brace with property \PS. Then $\overline{\operatorname{Soc}}(B)$ is a good ideal of $B$.
\end{thm}
\begin{proof}
By definitions and Corollary \ref{corextension}, we just need to show that $S=\operatorname{Soc}(B)$ is a good ideal of $B$. 

Let $G=N\rtimes_\lambda X$, where $N=(B,+)$ and $X=(B,\circ)$. Let $E$ be any finite subset of $S$. Since $G$ is an $FC$-group, the normal closure $F$ of $E$ (here we see $E$ as a subgroup of $N$) in $G$ is finitely generated. As a subset of $B$, $F$ is a finitely generated normal subgroup of $(B,+)$, is a subgroup of $(B,\circ)$, and is invariant under the action of $\lambda$. But $F$ is still contained in $\operatorname{Soc}(B)$, so, if $x\in F$ and $b\in B$, then \cite[Lemma 1.10]{Cedo} yields that $b\circ x\circ b^{-1}=\lambda_b(x)\in F$. It follows that $F$ is a normal subgroup of $(B,\circ)$ and hence $F$ is an ideal of $B$ containing $E$. Since $F$ is additively finitely generated, it is also finitely generated as a skew brace and the statement is proved.
\end{proof}

\medskip

In a similar way we define the {\it upper annihilator series} $\{\operatorname{Ann}_\alpha(B)\}_{\alpha\in\operatorname{Ord}}$ of $B$. The smallest ordinal number $a(B)$ such that $\operatorname{Ann}_{a(B)}(B)=\operatorname{Ann}_{a(B)+1}(B)$ is the {\it upper annihilator length}, the last term of the upper annihilator series is the {\it hyper-annihilator} and is denoted by $\overline{\operatorname{Ann}}(B)$. Again, every term of the upper annihilator series is an ideal and $\overline{\operatorname{Ann}}\big(B/\overline{\operatorname{Ann}}(B)\big)$ is zero. Moreover, for each ordinal number $\alpha$, $\operatorname{Ann}_\alpha(B)\subseteq\operatorname{Soc}_\alpha(B)$, so that $\overline{\operatorname{Ann}}(B)\subseteq\overline{\operatorname{Soc}}(B)$, and hence the hyper-annihilator is a good ideal by Theorem \ref{goodsocle}. Note that the latter inclusion can be strict even in the case of a periodic skew brace. For instance, let $B$ be the brace obtained as the semidirect product $C \rtimes A$, where $A, C$ are the unique braces of orders $3$ and $2$, respectively, and the action is the unique non-trivial action of $C$ on $A$ (see \cite{Okninski}). One can check that $\Soc_2(B)=B$, but the annihilator of $B$ is zero.

A skew brace $B$ such that $B=\operatorname{Ann}_n(B)$ for some non-negative integer $n$ is called {\it annihilator nilpotent} (of length $n$). In fact, the upper annihilator series can be thought as an analogue of the upper central series of a group using the notion of nilpotency introduced in \cite{Bonatto} (called ‘‘annihilator nilpotent’’ in \cite{JeVAV22x}). It is well known that in any $FC$-group the upper central length is bounded by $\omega$ and it turns out that a similar statement holds for skew braces having property \PS.

\begin{thm}\label{upperomega}
Let $B$ be a skew brace with property \PS. Then $B$ admits a series of ideals $$\{0\}\subseteq S_0(B)\subseteq S_1(B)\subseteq\ldots\subseteq S_n(B)\subseteq S_{n+1}(B)\subseteq \ldots S_\omega(B)=\overline{\operatorname{Ann}}(B)$$ such that:
\begin{itemize}
    \item $S_0(B)$ is torsion-free and $S_0(B)\subseteq\operatorname{Ann}(B)$; 
    \item for each $i\in\mathbb{N}\cup\{0\}$, $S_{i+1}(B)/S_i(B)\subseteq\operatorname{Ann}\big(B/S_i(B)\big)$ and $S_{i+1}(B)/S_i(B)$ is a restricted direct product of finite ideals of $B/S_i(B)$.
\end{itemize}

\noindent In particular, $a(B)\leq\omega$.
\end{thm}
\begin{proof}
By Theorem \ref{thgzglf}, we may assume that $B$ is locally finite. 
Theorem \ref{goodsocle} shows that every minimal ideal of $B$ contained in $\overline{\operatorname{Ann}}(B)$ is finite. Now we define a series of ideals of $B$ as follows. Let $S_0=\{0\}$. An easy application of Zorn's lemma yields that the ideal $S_1(B)=S(B)$ generated by all minimal ideals of $B$  which are contained in $\overline{\operatorname{Ann}}(B)$ is the direct product of some of them. If $\alpha$ is any ordinal number, we put $$S_{\alpha+1}(B)/S_\alpha(B)=S\big(B/S_\alpha(B)\big),$$ while if $\mu$ is any limit ordinal, we put $$S_\mu(B)=\bigcup_{\alpha<\mu} S_\alpha(B).$$ It is clear that each $S_\alpha(B)$ is an ideal of $B$ contained in $\overline{\operatorname{Ann}}(B)$.  Let $I$ be any finite ideal of $B$ contained in $\overline{\operatorname{Ann}}(B)$. We prove by induction on the cardinality of $I$ that $I$ is contained in $S_\omega(B)$. Clearly, $I$ contains a minimal ideal $J$ of $B$, so~$J$ is contained in $S_1(B)$ and $\big(I+S_1(B)\big)/S_1(B)$ has smaller order, so induction yields that it is contained in  $S_\ell\big(B/S_1(B)\big)$ for some $\ell$. Consequently, $I$ is contained in~$S_{\ell+1}(B)$ and we are done.

By Theorem \ref{goodsocle}, every finite subset of $\overline{\operatorname{Ann}}(B)$ is contained in some finite ideal, which means that it is also contained in $S_\omega(B)$ by the above claim. It follows that $S_\omega(B)=\overline{\operatorname{Ann}}(B)$.

Finally, we claim that $S_m(B)$ is contained in $\operatorname{Ann}_m(B)$ for every $m$: this will show that $\operatorname{Ann}_\omega(B)=S_\omega(B)=\overline{\operatorname{Ann}}(B)$ and complete the proof of the statement. In order to prove the claim, it is enough to show that $S_1(B)$ is contained in $\operatorname{Ann}(B)$ so that every minimal ideal $I$ of $B$ which is contained in $\overline{\operatorname{Ann}}(B)$ is also contained in $\operatorname{Ann}(B)$. Let $\alpha$ be the smallest ordinal number such that $I\cap\operatorname{Ann}_\alpha(B)$ is non-trivial (such an ordinal exists because $I$ is contained in $\overline{\operatorname{Ann}}(B)$); in particular, \mbox{$I\subseteq\operatorname{Ann}_\alpha(B)$.} Of course, $\alpha$ is not limit and not $0$, so $I\cap \operatorname{Ann}_{\alpha-1}(B)=\{0\}$. Let $x\in I$. For each $b\in B$, we have that $x\ast b,b\ast x,[x,b]_+,[x,b]_\circ$ are elements of $I$ and are trivial modulo $\operatorname{Ann}_{\alpha-1}(B)$. Therefore $$x\ast b,b\ast x,[x,b]_+,[x,b]_\circ\in I\cap \operatorname{Ann}_{\alpha-1}(B)=\{0\}$$ and hence $I\subseteq\operatorname{Ann}(B)$. This completes the proof.
\end{proof}

\medskip

\begin{rem}
If we replace the operator ‘‘$\operatorname{Ann}$’’ with the operator ‘‘$\operatorname{Soc}$’’, an analogue of Theorem~\ref{upperomega} holds with the same proof.
\end{rem}

\medskip

Now, using the concept of good ideals of a skew brace with property \PS, we are able to locally characterise those ideals which are contained in the hyper-annihilator.

\begin{thm}\label{LNth}
Let $B$ be a skew brace with property \PS and let $I$ be an ideal of $B$. Then the following conditions are equivalent:
\begin{itemize}
    \item[\textnormal{(1)}] $I$ is good, and if $A$ is a finitely generated sub-skew brace of $B$, then $A\cap I$ is contained in $\overline{\operatorname{Ann}}(A)$.
    \item[\textnormal{(2)}] $I\subseteq\overline{\operatorname{Ann}}(B)$.
\end{itemize}
\end{thm}
\begin{proof}
First, we prove (2) implies (1). Let $A$ be any finitely generated sub-skew brace of $B$. It follows from Theorem \ref{goodsocle} that $I$ is good. Moreover, for each ordinal $\alpha$, we have that $$\operatorname{Ann}_\alpha(B)\cap A\subseteq \operatorname{Ann}_\alpha(A),$$ so $$I\cap A\subseteq\overline{\operatorname{Ann}}(B)\cap A\subseteq\overline{\operatorname{Ann}}(A)$$ and we are done.

\medskip

Now, we prove (1) implies (2). Let $J$ be a finitely generated ideal of $B$ contained in $I$, so by the hypothesis, $J$ is finitely generated as a skew brace. Let $T=T_+(J)$. Then $T$ is a finite ideal of $B$ (see Theorem \ref{theoperiodicfc}). Moreover, $J/T$ is contained in $\Ann(B/T)$ by Corollary \ref{idealtorsionfreeskewlefttrivial}, so it only remains to show that $T$ is contained in $\overline{\operatorname{Ann}}(B)$. To this end we may assume~$I$ is finite and non-zero. Let $F$ be a minimal ideal of $B$ contained in~$I$, so~$F$ is finite. Then, it is certainly possible to find a finitely generated sub-skew brace $E$ of $B$ such that~$F$ is a minimal ideal of $E$. But $F\subseteq E\cap I$ is contained in the hyper-annihilator of $E$, so an argument we employed at the end of the proof of Theorem \ref{upperomega} yields that $F\subseteq\operatorname{Ann}(E)$. Since every sub-skew brace containing $E$ has the same property as $E$, it follows that $F\subseteq\operatorname{Ann}(B)$. Let $K=I\cap \operatorname{Ann}(B)$; in particular, $K$ is a finite ideal of $B$. Now, $I/K$ satisfies the hypotheses of (1) in~$B/K$, so by induction, we get that $I/K$ lies in the hyper-annihilator of~$B/K$. Then~$I$ lies in the hyper-annihilator of $B$ and the statement is proved.
\end{proof}

\medskip

A skew brace is {\it locally annihilator-nilpotent} if every finitely generated sub-skew brace is annihilator nilpotent. As a consequence of the above result, we have an interesting characterisation of locally annihilator-nilpotent skew braces in the universe of good skew braces satisfying property \PS. This corresponds to the fact that an $FC$-group is locally nilpotent if and only if it is hypercentral.

\begin{cor}\label{LN}
Let $B$ be a good skew brace having property \PS. Then the following conditions are equivalent:
\begin{itemize}
    \item[\textnormal{(1)}] $B$ is locally annihilator-nilpotent.
    \item[\textnormal{(2)}] $B=\overline{\operatorname{Ann}}(B)$.
\end{itemize}
\end{cor}

\begin{rem}
In Theorem~\ref{LNth}, the operator ‘‘$\operatorname{Ann}$’’ can be replaced by the operator ‘‘$\operatorname{Soc}$’’.
\end{rem}

\medskip

We now give an application of Theorem \ref{periodicreduction} and Theorem \ref{goodsocle}. As already noticed a theorem of Issai Schur states that if a group is finite over its centre, then also its commutator subgroup is finite; it is also well known that the converse to this statement holds whenever the group is finitely generated. Analogous results for skew braces have been proved in \cite{MR4256133} (see also our Lemma \ref{lemma1}). Schur's theorem is a very particular case of a more general result of Baer \cite{Baer}, which states that if a group is finite over its $n$-th centre, then the $(n+1)$-term of its lower central series is finite. The analogue of this result is true for skew braces with property~\PS, it is not clear if this is the case without such an assumption. 

\begin{thm}\label{BaerThm}
Let $B$ be a skew brace with property \PS and let $n$ be a positive integer. If $B/\operatorname{Ann}_n(B)$ is finite, then $B$ has a finite ideal $I$ such that $B/I$ is annihilator nilpotent of length $n$.
\end{thm}
\begin{proof}
Theorems \ref{extension} and \ref{goodsocle} show that $B$ is good, so every factor group of $B$ is good and Theorem \ref{periodicreduction} shows that we may assume $B$ periodic. Let $F$ be a complete set of representatives of $A=\operatorname{Ann}_n(B)$ in $(B,+)$. Then the ideal $I$ generated by $F$ is finite and $B=A+I$ and the statement is proved.
\end{proof}

\medskip

The theorems of Schur and Baer can be generalised in many other ways. For example, it has been proved in \cite{dg} that if $G$ is finite over its hypercentre, then $G$ has a finite normal subgroup whose factor group is hypercentral. The proof of the above result also makes it possible to prove a corresponding statement in the case of skew braces having property \PS.

\begin{thm}
Let $B$ be a skew brace with property \PS. If~$B$ is finite over the hyper-annihilator, then $B$ has a finite ideal $I$ such that $B/I$ coincides with its hyper-annihilator.
\end{thm}

\medskip

We end this section dealing with the radical theory of skew braces having property \PS. The radical theory of a skew brace has been introduced in \cite{MR4256133} as follows. Let $B$ be a skew brace. Then $\operatorname{Rad}(B)$ is defined as the intersection of all maximal ideals of $B$. Theorem 3.10 of that paper shows that if $B$ is {\it Artinian} (that is, $B$ satisfies the minimal condition on ideals), then $B/\operatorname{Rad}(B)$ is isomorphic to a direct product of finitely many simple skew braces. We can prove something similar in the universe of skew braces with property \PS. First, we need the following analogue of a classical argument of Remak. In what follows, ‘‘$\operatorname{Dr}$’’ denotes the restricted direct product.

\begin{lem}\label{remak}
Let $B$ be a skew brace with property \PS and assume $B=\operatorname{Dr}_{i\in\mathcal{I}}B_i$, where $B_i$ is a simple skew brace which is not the trivial brace. If $I$ is an ideal of $B$, then $I$ is the direct product of certain $B_i$.
\end{lem}
\begin{proof}
Clearly, we may factor out any $B_i$ contained in $I$, assuming therefore that $I\cap B_i=\{0\}$, for all $i$. Moreover, since none of the $B_i$ is the trivial brace, we have that $\Ann(B)=\{0\}$.

Now, $[I,B_i]_+\subseteq I\cap B_i=\{0\}$ and similarly, $[I,B_i]_\circ=I\ast B=B\ast I=\{0\}$. Thus $I\subseteq\Ann(B)=\{0\}$ and the statement is proved.
\end{proof}

\begin{thm}\label{radical}
Let $B$ be a skew brace with property \PS and $R=\operatorname{Rad}(B)$. Then $B/R$ can be embedded into the direct product $C\times D$, where 
\begin{itemize}
    \item[\textnormal{(1)}] $C$ is the restricted direct product of all simple homomorphic images of $B$ which are not trivial braces, and
    \item[\textnormal{(2)}]  $D$ is a trivial brace which is an unrestricted direct product of trivial braces of prime order.
\end{itemize}

Moreover, $B/A$ is isomorphic to $C$, where $A/R=\Ann(B/R)$.
\end{thm}
\begin{proof}
We may clearly assume $R=\{0\}$. Let $\mathcal{I}$ be the set of all maximal ideals of~$B$. Let $b\in B$ and let $\mathcal{J}_b$ be the subset of $\mathcal{I}$ made by all ideals $J$  such that $$b+J\notin\operatorname{Ann}(B/J)$$ and $B/J$ is not the trivial brace. We claim that $\mathcal{J}_b$ is finite. Suppose this is not the case. Since $B$ has property \PS, we have that the number of conjugates of $b$ in $B$ is finite, say $m_b$. Let $I_1,\ldots, I_{m_b}$ be distinct elements of $\mathcal{J}_b$ such that $$I_1\supset I_1\cap I_2\supset \ldots\supset I_1\cap\ldots\cap I_{m_b}=J_b$$ (the existence of these ideal is guaranteed by Lemma \ref{remak}). The argument employed in the proof of \cite[Theorem 3.10]{MR4256133} shows that the map $$\varphi_b:\, a+J_b\in B/J_b\mapsto \big(a+I_1,\ldots,a+I_{m_b}\big)\in B/I_1\times\ldots\times B/I_{m_b}$$ is an isomorphism. However, it is easy to see that the number of conjugates of~$\varphi_b(b)$ is at least $m_b+1$, a contradiction. Thus, $\mathcal{J}_b$ is finite.

Let $\mathcal{T}$ be the subset of $\mathcal{I}$ made by all maximal ideals $I$ of $B$ such that $B/I$ is a trivial brace. The previous claim shows that for each $b\in B$ there are finitely many elements $I$ in $\mathcal{I}\setminus\mathcal{T}$ such that $b\notin I$. Let $N$ be the intersection of all members of~$\mathcal{T}$, so $B/N$ is isomorphic to an unrestricted direct product of trivial braces of prime order. It follows that the map $$\varphi:\, b\in B\mapsto \big((b+I)_{I\in\mathcal{I}\setminus\mathcal{T}}, b+N\big)\in\underset{I\in \mathcal{I}\setminus{\mathcal{T}}}{\operatorname{Dr}}(B/I)\times B/N$$ is a monomorphism. 

Finally, let $M$ be the intersection of all ideals in $\mathcal{I}\setminus\mathcal{T}$. Since $M\subseteq\varphi^{-1}(B/N)$, it follows that $M$ is contained in $\Ann(B)$. If we let $$\psi:\, B\longrightarrow \underset{I\in \mathcal{I}\setminus{\mathcal{T}}}{\operatorname{Dr}}(B/I)$$ be the natural projection, we see that $M=\Ker(\psi)$. We need to show that $\psi$ is surjective.

Let $I\in\mathcal{I}\setminus\mathcal{T}$ and take $x\in B\setminus I$. Let $J_1,\ldots,J_n$ be the elements of $\mathcal{I}\setminus\mathcal{T}$ such that $J_1=I$. Since $B/J_1$ is a simple skew brace which is not the trivial brace, we can find $c\in B$ such that one of the elements $x\ast c,c\ast x,x^{\circ, c}, x^{+,c}$, call it $u$, is not contained in $J_1$. As we notice above, there is $y\in B$ such that $y+J_1=c+J_1$ and $y\in J_2\cap\ldots\cap J_n$. It follows that $u$ belongs to $N$ and to all elements of $\mathcal{I}\setminus\mathcal{T}$ except $J_1$. The ideal generated by $u+J_1$ in $B/J_1$ is $B/J_1$. As $I$ was arbitrary, shows that $\psi$ is surjective and completes the proof.
\end{proof}

\medskip

With respect to the previous result, we observe that a simple skew brace with property \PS must be locally finite (see Theorems \ref{theoperiodicfc} and \ref{corschurlf}).

\section{Skew braces with property \PBS}\label{bfc}

\noindent Recall that a group $G$ is a {\it $BFC$-group} if there is a positive integer $n$ such that every element of $G$ has at most $n$ conjugates in $G$. The main result of this section can be understood as a brace-theoretic analogue of a well-known result of B.H. Neumann stating that a group $G$ is a $BFC$-group if and only if $[G,G]$ is finite. 

\begin{defn}\label{defbfc}
An arbitrary skew brace has property \PBS if there is a positive integer~$n$ such that $$\operatorname{sup}\big\{|(B,+):\operatorname{Fix}^r(x)\cap C_x^+|, |(B,\circ):\operatorname{Fix}^l(x)\cap C_x^\circ|\,: x\in B\big\}\leq n.$$
\end{defn}

 Of course, every skew brace $B$ such that $|B/\operatorname{Ann}(B)|<\infty$ satisfies~\PBS, so the next result can also be seen as a generalisation of the first half of~\cite[Theorem~5.4]{MR4256133}. 
 

\begin{thm}\label{thm1}
Let $B$ be a skew brace. Then $B$ has property \PBS if and only if $B^{(2)}$ and $[B,B]_+$ are finite.
\end{thm}
\begin{proof}
If $B^{(2)}$ and $[B,B]_+$ are finite, then also $[B,B]_\circ$ is finite and the statement is proved. Assume conversely that $B$ has property \PBS.  Therefore there is a positive integer $n$ such that $$\operatorname{sup}\big\{|(B,+):\operatorname{Fix}^r(x)\cap C_x^+|, |(B,\circ):\operatorname{Fix}^l(x)\cap C_x^\circ|\,: x\in B\big\}\leq n.$$ Let $G=N\rtimes_\lambda X$, where $N=(B,+)$ and $X=(B,\circ)$. We claim that $G$ is a $BFC$-group. Let $g\in N$. Then $g$ has at most $n$ conjugates in $N$. It follows from Equation \eqref{equationfinale} that every element of $N$ has at most $n$ conjugates under the action of $X$, which means that every element of $N$ has at most $n^2$ conjugates in $G$. Similarly, any element of $X$ has at most $n^2$ conjugates in $G$. Finally, let $g$ be an arbitrary element of $G$. Then we can write $g=g_1g_2$, where $g_1 = (0,x)$, $g_2=(y,0)$ $x,y \in B$. Since $$|G:C_G(g_1)\cap C_G(g_2)|\leq |G:C_G(g)|,$$ it follows that $|G:C_G(g)|\leq n^4$ and the claim is proved. 

Now, by Neumann's theorem, we have that $[G,G]$ is finite; in particular $[B,B]_+$ is finite. The fact that $B^{(2)}$ is finite can be deduced from Equation~\eqref{equationfinale}. The statement is proved.
\end{proof}

\medskip

\medskip

Finally, we show how skew braces with property \PBS come into play in the context of the YBE. Let $(X,r)$ be a set-theoretic non-degenerate solution of the YBE. Recall first that the {\it \textnormal(left\textnormal) derived} solution $(X,r')$ of $(X,r)$ is the non-degenerate solution defined as follows: $r'(x,y)=(y,\eta_y(x))$, where $$\eta_y(x)=\sigma_y\big(\tau_{\sigma_x^{-1}(y)}(x)\big).$$ It is easy to see that the structure group $G(X,r')$ coincides with $A(X,r)$ (see \cite[Theorem 2.4]{ESG01}).

\begin{defn}\label{defderivedindecomposable}
    A non-degenerate solution $(X,r)$ of the YBE is said to be \emph{derived-indecomposable} if its derived solution $(X,r')$ is indecomposable.  
\end{defn}

One of the main objects controlling the decomposability of a non-degenerate solution $(X,r)$ is its \emph{permutation group}, \hbox{i.e.} the subgroup $\mathcal{G}(X,r)=\langle\sigma_x,\tau_y\,\colon\, x,y\in X\rangle$ of $\operatorname{Sym}(X)$. Indeed, by \cite[Proposition 6.6]{Cedo}, $(X,r)$ is indecomposable if and only if $\mathcal{G}(X,r)$ acts transitively on $X$. Moreover, by the proof  of \cite[Theorem 2.2]{CaMaSt22}, one obtains that the action of the permutation group is the same as the action of the group $\langle\sigma_x,\eta_x\,\colon\, x\in X\rangle$. Thus, if $(X,r')$ is indecomposable, then $\mathcal{G}(X,r')=\langle \eta_x \colon x\in X\rangle$ is transitive, so that $\langle\sigma_x,\eta_x\,\colon\, x\in X\rangle$ is transitive and hence $(X,r)$ is indecomposable. More specifically, there exists a unique derived-indecomposable involutive solution, namely the involutive solution defined on the set of one element, but
among non-involutive ones there are large classes of solutions which are derived-indecomposable (see for instance \cite[Appendix A]{LeVe19}). However, even in the non-involutive case, indecomposable does not imply derived-indecomposable, as shown by the following example. Let $X=\{1,2,3,4\}$ and define
\begin{align*}
\sigma_x=\begin{cases}(1\ 4\ 2 \ 3) \text{ if } x=1,2\\ 
(1\ 3 \ 2 \ 4) \text{ if } x =3,4 \qquad \text{and}\qquad \tau_x= (1\ 3 )(2\ 4).
\end{cases}\end{align*}
It follows that $r(x,y) = (\sigma_x(y),\rho_y(x))$ is an indecomposable solution but its derived solution defined by $$\eta_x(y)=\begin{cases}
    (3\ 4) \text{ if } x=1,2\\ (1\ 2) \text{ if } x=3,4
\end{cases}$$
is decomposable with orbits $\{1,2\}$ and $\{3, 4\}$.

\smallskip

Similarly to Corollary \ref{corinvolutive}, in the case of skew braces with property \PS, the converse holds.

\begin{thm}\label{indecomposableprima}
Let $(X,r)$ be a non-degenerate finite solution of the YBE such that $G(X,r)$ has property \PS. Then $(X,r)$ is indecomposable if and only if $(X,r)$ is derived-indecomposable.
\end{thm}
\begin{proof}
By our remark above, we only need to show that if $(X,r)$ is indecomposable, then $(X,r')$ is such. Therefore assume $(X,r)$ is indecomposable. It follows from \cite[Proposition 7.3]{LeVe19} that $G(X,r)/T$ is infinite cyclic, where $T$ is the subgroup of all periodic elements of $G(X,r)$. But $G(X,r)*G(X,r)$ is periodic by Corollary~\ref{corschurlf}, so also $A(X,r)/T$ is infinite cyclic. But $A(X,r) = G(X,r')$, so $G(X,r')/T$ is infinite cyclic. A further application of \cite[Proposition 7.3]{LeVe19} shows that $(X,r')$ is indecomposable.
\end{proof}

\medskip

Our last result shows that the derived-indecomposability has a very strong impact on the behaviour of the structure skew brace. 

\begin{thm}\label{finalthmBrace}
Let $(X,r)$ be a non-degenerate finite solution of the YBE. If $(X,r)$ is derived-indecomposable, then the  skew brace $G(X,r)$ has property \PBS.
\end{thm}
\begin{proof}
Let $A$ be the additive group of skew brace $G(X,r)$. It follows from \cite[The\-o\-rem~6.2]{MR4256133} that $G(X,r)\ast G(X,r)$ is finite. On the other hand, $A$ is finite over its centre (this is seen by adapting the proof of \cite[Theorem~2.7]{structureMonoid}) yields that $A$ is finite over its centre, so Schur's theorem gives that $[A,A]_+$ is finite. Then $G(X,r)$ has property \PBS by Theorem \ref{thm1}. 
\end{proof}

\section*{Acknowledgments}

The authors are members of the non-profit association ``Advances in Group Theory and Applications'' (www.advgrouptheory.com). Colazzo is supported by the Engineering and Physical Sciences Research Council [grant number EP/V005995/1]. Ferrara and Trombetti are supported by GNSAGA (INdAM). Ferrara is also supported by a grant from the University of Campania ``Luigi Vanvitelli'', in the framework of the projects GoAL (V:ALERE 2019) and HELM (V:ALERE 2020).

We thank the referees for their useful comments and suggestions. In particular, we are indebted to the referees for Lemma~\ref{remref2}, which showed that the ‘‘periodicity’’ assumption from Theorem~\ref{annsubgroup} could be dropped out.

We thank L. Vendramin for the useful discussion and comments on an earlier version of the paper.

\section{Statements and Declarations}

\noindent{\bf Conflict of interest}\quad The authors declare that they have no conflict of interest.

\smallskip

\noindent{\bf Data availability}\quad Data sharing is not applicable to this article as no new data were created or analysed in this study.

\bibliographystyle{abbrv}
\bibliography{refs}
\end{document}